\documentclass[12pt, reqno]{amsart}
\usepackage{amsmath, amsthm, amscd, amsfonts, amssymb, graphicx, color}
\usepackage[bookmarksnumbered, colorlinks, plainpages]{hyperref}

\input{mathrsfs.sty}

\newtheorem{theorem}{Theorem}[section]
\newtheorem{lemma}[theorem]{Lemma}
\newtheorem{proposition}[theorem]{Proposition}
\newtheorem{corollary}[theorem]{Corollary}
\theoremstyle{definition}

\newtheorem{example}[theorem]{Example}

\theoremstyle{remark}

\numberwithin{equation}{section}
\begin{document}

\title { Further refinements of the Cauchy--Schwarz inequality for matrices}

\author[M. Bakherad]{Mojtaba Bakherad }

\address{ Department of Mathematics, Faculty of Mathematics, University of Sistan and Baluchestan, Zahedan, Iran.}
\email{mojtaba.bakherad@yahoo.com; bakherad@member.ams.org}


\begin{abstract}
Let $A, B$ and $X$ be $n\times n$ matrices such that $A, B$ are positive semidefinite. We present some refinements of the matrix Cauchy-Schwarz
 inequality by using  some
integration techniques and various refinements of the Hermite--Hadamard
inequality. In particular, we  establish the inequality
\begin{align*}
|||\,|A^{1\over2}XB^{1\over2}|^r|||^2&\leq|||\,|A^{t}XB^{1-s}|^r||| \,\,\,|||\,|A^{1-t}XB^{s}|^r|||\\&
\leq\max \{|||\,|AX|^r||| \,\,\,|||\,|XB|^r|||,|||\,|AXB|^r||| \,\,\,|||\,|X|^r|||\},
 \end{align*}
 where $s,t\in[0,1]$ and $r\geq0$.

\end{abstract} \maketitle
\section{Introduction and preliminaries}
\noindent Let $\mathcal{M}_n$ be the $C^*$-algebra of all
$n\times n$ complex matrices. For Hermitian matrices $A, B\in
\mathcal{M}_n$, we write  $A\geq 0$ if $A$ is positive semidefinite,
$A>0$ if $A$ is positive definite, and $A\geq B$ if $A-B\geq0$. We
use $\mathcal{S}_n$ for the set of positive semidefinite matrices
and $\mathcal{P}_n$ for the set of positive definite matrices in  $\mathcal{M}_n$.
A norm $|||\,.\,|||$ is called unitarily invariant norm if
$|||UAV|||=|||A|||$ for all $A\in\mathcal{M}_n$ and  all unitary
matrices $U, V\in\mathcal{M}_n$. The numerical range of $A\in\mathcal{M}_n$ is
$W(A)=\{\langle Ax, x\rangle: x\in\mathbb{C}^n, \|x\|=1 \}$ and
the numerical radius of  $A$ is defined by
$\omega(A)=\sup\{|\langle Ax, x\rangle|: x\in\mathbb{C}^n, \|x\|=1 \}$.
 It is well-known \cite{gus} that $\omega(\,\cdot\,)$ is a weakly unitarily invariant   norm on $\mathcal{M}_n$, that is $\omega(U^*AU)=\omega(A)$ for every unitary $U\in\mathcal{M}_n$. The  Hadamard product (Schur product) of two matrices $A, B\in\mathcal{M}_n$  is the
matrix $A\circ B$ whose $(i, j)$ entry is $a_{ij}b_{ij}\,\,(1\leq i,j \leq n)$.
 The Schur multiplier  operator $S_A$ on  $\mathcal{M}_n$ is defined by $S_A=A\circ X\,\,(X\in\mathcal{M}_n)$.
 The induced norm of $S_A$ with respect to the spectral norm  is  $\|S_A\|=\sup_{X\not=0}{\|S_A(X)\|\over \|X\|}=\sup_{X\not=0}{\|A\circ X\|\over \|X\|},$
 and the induced norm of $S_A$ with respect to numerical radius norm will be denoted by
 \begin{align*}
\|S_A\|_\omega=\sup_{X\not=0}{\omega(S_A(X))\over \omega(X)}=\sup_{X\not=0}{\omega(A\circ X)\over \omega(X)}.
 \end{align*}

  A continuous real valued function $f$ on an interval $J\subseteq \mathbb{R}$ is called operator monotone  if $A\leq B$ implies $f(A)\leq f(B)$  for all $A, B\in\mathcal{M}_n$ with spectra in $J$.
Recall that a  real valued function $F$ defined on $J_1\times J_2$ is called convex if
 \begin{align*}
F(\lambda x_1+(1-\lambda )x_2,\lambda y_1+(1-\lambda )y_2)\leq \lambda F(x_1,y_1)+(1-\lambda)F(x_2,y_2)
 \end{align*}
 for all $x_1,x_2\in J_1, y_1,y_2\in J_2$ and $\lambda\in[0,1]$. \\
  For two  sequences $a=(a_1, a_2, \cdots,a_n)$ and $b=(b_1, b_2,\cdots, b_n)$ of real numbers, the classical Cauchy-Schwarz inequality  states that
 \begin{align*}
\left(\sum_{j=1}^n a_jb_j  \right)^2\leq\left(\sum_{j=1}^n a_j^2 \right)\left(\sum_{j=1}^n b_j^2  \right)
 \end{align*}
 with equality if and only if the sequences $a$  and $b$ are proportional \cite{pec}.
Horn and Mathias \cite{horn1} gave a   matrix Cauchy-Schwarz inequality as follows
 \begin{align*}
||| \,|A^*B|^r|||^2\leq |||\,(AA^*)^r|||\,\,\,|||\,(BB^*)^r|||\qquad(A, B, X\in\mathcal{M}_n,r\geq0).
 \end{align*}
 Bhatia and Davis \cite{davis} showed that
 \begin{align}\label{a24}
||| \,|A^*XB|^r|||^2\leq |||\,|AA^*X|^r|||\,\,\,|||\,|XBB^*|^r|||\qquad(A, B, X\in\mathcal{M}_n,r\geq0),
 \end{align}
 which is  equivalent to
 \begin{align}\label{d41}
||| \,|A^{1\over2}XB^{1\over2}|^r|||^2\leq |||\,|AX|^r|||\,\,\,|||\,|XB|^r|||\qquad(A, B\in\mathcal{S}_n, X\in\mathcal{M}_n,r\geq0).
 \end{align}
 In \cite{hiai} it is proved that the function $f(t)=|||\,|A^{t}XB^{1-t}|^r||| \,\,\,|||\,|A^{1-t}XB^{t}|^r|||$   is convex on the interval $[0,1]$, when  $A, B\in\mathcal{S}_n,$ $ X\in\mathcal{M}_n$  and attains its minimum at $t={1\over2}$. In view of the fact that the  function $f$ is decreasing on the interval $[0,{1\over2}]$ and increasing on the interval $[{1\over2},1]$.  In particular,   we have a refinement of the Cauchy-Schwarz inequality \cite{hiai} as follows
\begin{align}\label{a6}
|||\,|A^{1\over2}XB^{1\over2}\,|^r|||^2\leq|||\,|A^{\mu}XB^{1-\mu}|^r||| \,\,\,|||\,|A^{1-\mu}XB^{\mu}|^r|||\leq |||\,|AX|^r|||\,\,\,|||\,|XB|^r|||,
 \end{align}
where  $A, B\in\mathcal{S}_n,$ $ X\in\mathcal{M}_n$ and $\mu\in[0,1]$.\\
  Applying the convexity of the function $f(t)=|||\,|A^{t}XB^{1-t}|^r||| \,\,\,|||\,|A^{1-t}XB^{t}|^r|||$ $(t\in[0,1])$, we show some refinements of inequality \eqref{a6}. we also show the convexity of the function $f(s,t)=|||\,|A^{s}XB^{1-t}|^r||| \,\,\,|||\,|A^{1-s}XB^{t}|^r||| $ and present some other refinements of inequality \eqref{a6}.  In the last section we show some related numerical radius inequalities.
\section{ Norm inequality involving the Cauchy-Schwarz }
In this section, we establish some refinements of inequality \eqref{a6}. To this end, we need the following Hermite-Hadamard inequality.
\begin{lemma}\cite{feng}\label{a7}
Let g be a real-valued convex function on
$[a, b]$. Then
\begin{align*}
g\left({a+b\over2}\right)\leq {1\over b-a}\int_a^b g(s)ds\leq{1\over4}\Big[g(a)+2g\Big({a+b\over2}\Big)+g(b)\Big]\leq {g(a)+g(b)\over2}.
 \end{align*}
 \end{lemma}
  Applying Lemma \ref{a7} we have following result.
\begin{proposition}\label{a8}
Suppose that $A, B\in\mathcal{S}_n$, $ X\in \mathcal{M}_n$ and $r\geq0$. Then
\begin{align*}
|||\,|A^{1\over2}XB^{1\over2}|^r|||^2&\leq
{1\over|1-2\mu|}\left|\int_\mu^{1-\mu}
|||\,|A^{s}XB^{1-s}|^r|||  \,\,\,|||\,|A^{1-s}XB^{s}|^r|||\,ds \right|\\&\leq
{1\over2}\Big[|||\,|A^{1\over2}XB^{1\over2}|^r|||^2+|||\,|A^{\mu}XB^{1-\mu}|^r|||  \,\,\,|||\,|A^{1-\mu}XB^{\mu}|^r|||\Big]\\&\leq
|||\,|A^{\mu}XB^{1-\mu}|^r|||  \,\,\,|||\,|A^{1-\mu}XB^{\mu}|^r|||
 \end{align*}
for all $0\leq\mu\leq1$ and all unitarily invariant norms $|||\,.\,|||$.
 \end{proposition}
 \begin{proof}
  Let $f(t)=|||\,|A^{t}XB^{1-t}|^r||| \,\,\,|||\,|A^{1-t}XB^{t}|^r|||$. First assume that $0\leq\mu<{1\over2}$.  It  follows from  Lemma  \ref{a7} that
 \begin{align*}
f\left({\mu+1-\mu\over2}\right)&\leq {1\over 1-2\mu} \int_\mu^{1-\mu}f(s)ds\\&\leq{1\over4}\Big[{f(\mu)+2f\left({\mu+1-\mu\over2}\right)+f(1-\mu)}\Big]\\&\leq {f(1-\mu)+f(\mu)\over2},
 \end{align*}
 whence
 \begin{align*}
f\left({1\over2}\right)\leq {1\over 1-2\mu} \int_\mu^{1-\mu}f(s)ds\leq {1\over2}\Big[{f(\mu)+f({1\over2})}\Big]\leq {f(\mu)}.
 \end{align*}
 Hence
 \begin{align}\label{a10}
|||\,|A^{1\over2}XB^{1\over2}|^r|||^2&\leq{1\over1-2\mu}\int_\mu^{1-\mu}
|||\,|A^{1-s}XB^{s}|^r|||\,\,\,|||\,|A^{s}XB^{1-s}|^r|||\, ds \nonumber \\&\leq
{1\over2}\Big[|||\,|A^{1\over2}XB^{1\over2}|^r|||^2+|||\,|A^{\mu}XB^{1-\mu}|^r|||  \,\,\,|||\,|A^{1-\mu}XB^{\mu}|^r|||\Big]\nonumber\\&\leq|||\,|A^{\mu}XB^{1-\mu}|^r||| \,\,\,|||\,|A^{1-\mu}XB^{\mu}|^r|||.
\end{align}
Now, assume that ${1\over2}<\mu\leq1$. By the  symmetry property of  \eqref{a10} with respect to  $\mu$, if we replace $\mu$  by $1-\mu$, then
\begin{align}\label{a11}
|||\,|A^{1\over2}XB^{1\over2}|^r|||^2&\leq{1\over2\mu-1}\int_{1-\mu}^\mu
|||\,|A^{1-s}XB^{s}|^r|||\,\,\,|||\,|A^{s}XB^{1-s}|^r|||\, ds \nonumber\\&\leq
{1\over2}\Big[|||\,|A^{1\over2}XB^{1\over2}|^r|||^2+|||\,|A^{\mu}XB^{1-\mu}|^r|||  \,\,\,|||\,|A^{1-\mu}XB^{\mu}|^r|||\Big]\nonumber \\&\leq|||\,|A^{\mu}XB^{1-\mu}|^r||| \,\,\,|||\,|A^{1-\mu}XB^{\mu}|^r|||.
\end{align}
Since $\lim_{\mu\rightarrow{1\over2}}{1\over|2\mu-1|}\left|\int_\mu^{1-\mu}
|||\,|A^{s}XB^{1-s}|^r||| \,\,\, |||\,|A^{1-s}XB^{s}|^r|||\,ds \right|=|||\,|A^{1\over2}XB^{1\over2}|^r|||^2$,
 inequalities  \eqref{a10} and \eqref{a11} yield  the  desired result.
\end{proof}
 Now, we show  the convexity of the   function $$F(s,t)=|||\,|A^{1-t}XB^{1+s}|^r|||\,\,\,|||\,|A^{1+t}XB^{1-s}|^r|||$$ and  we use  the convexity of $F$ to prove some  Cauchy-Schwarz type inequalities.
\begin{theorem}\label{a5}
Suppose that $A, B\in\mathcal{S}_n$, $ X\in \mathcal{M}_n$ and $r\geq0$. Then  the function
\begin{align*}
F(s,t)=|||\,|A^{1-t}XB^{1+s}|^r|||\,\,\,|||\,|A^{1+t}XB^{1-s}|^r|||
 \end{align*}
is convex on  $[-1,1]\times[-1,1]$ and attains its minimum at $(0,0)$.
 \end{theorem}
 \begin{proof}
 The function $F$ is continuous and $F(s,t)=F(-s,-t)\,\,(s,t\in[0,1])$. Thus it is enough to show that
 \begin{align*}
F(s_1,t_1)\leq{1\over2}[F(s_1+s_2,t_1+t_2)+F(s_1-s_2,t_1-t_2)],
 \end{align*}
 where $s_1\pm s_2,t_1\pm t_2\in[-1,1]\times[-1,1]$.\\
  Let $s_1\pm s_2,t_1\pm t_2\in[-1,1]\times[-1,1]$. Applying inequality \eqref{a24} we obtain
 \begin{align}\label{a213}
|||\,|A^{1-t_1}X&B^{1+s_1}|^r|||=|||\,|A^{t_2}\big(A^{1-t_1-t_2}XB^{1+s_1-s_2}\big)B^{s_2}|^r|||\nonumber
\\&\leq \left\{|||\,|A^{1-(t_1-t_2)}XB^{1+(s_1-s_2)}|^r|||\,\,\,|||\,|A^{1-(t_1+t_2)}XB^{1+(s_1+s_2)}|^r|||\right\}^{1/2}
 \end{align}
 and
 \begin{align}\label{a223}
|||\,|A^{1+t_1}X&B^{1-s_1}|^r|||=|||\,|A^{t_2}\big(A^{1+t_1-t_2}XB^{1-s_1-s_2}\big)B^{s_2}|^r|||\nonumber
\\&\leq \left\{|||\,|A^{1+(t_1+t_2)}XB^{1-(s_1+s_2)}|^r|||\,\,\,|||\,|A^{1+(t_1-t_2)}XB^{1-(s_1-s_2)}|^r|||\right\}^{1/2}.
 \end{align}
 Applying \eqref{a213}, \eqref{a223} and the arithmetic-geometric mean inequality we get
 \begin{align*}
F(s_1,t_1)&=|||\,|A^{1-t_1}XB^{1+s_1}|^r|||\,\,\,|||\,|A^{1+t_1}XB^{1-s_1}|^r|||\\&\leq [F(s_1+s_2,t_1+t_2)F(s_1-s_2,t_1-t_2)]^{1/2}
\\&\leq{1\over2}[F(s_1+s_2,t_1+t_2)+F(s_1-s_2,t_1-t_2)].
 \end{align*}
   \end{proof}

 \begin{corollary}
Suppose that $A, B\in\mathcal{S}_n$, $ X\in \mathcal{M}_n$   and $r\geq0$. Then
\begin{align*}
|||\,|A^{1\over2}XB^{1\over2}|^r|||^2&\leq|||\,|A^{t}XB^{1-s}|^r||| \,\,\,|||\,|A^{1-t}XB^{s}|^r|||\\&
\leq\max \{|||\,|AX|^r||| \,\,\,|||\,|XB|^r|||,|||\,|AXB|^r||| \,\,\,|||\,|X|^r|||\},
 \end{align*}
 where $s,t\in[0,1]$.
 \end{corollary}
 \begin{proof}
  If we replace  $s$, $t$, $A, B$ by ${2s-1}$, ${2t-1}$, $A^{1\over2}, B^{1\over2}$, respectively, in Theorem \ref{a5}, we get the function $G(s,t)=|||\,|A^{t}XB^{1-s}|^r||| \,\,\,|||\,|A^{1-t}XB^{s}|^r|||$ is convex  on  $[0,1]\times[0,1]$ and attains its minimum at $({1\over2},{1\over2})$. Hence
\begin{align*}
|||\,|A^{1\over2}XB^{1\over2}|^r|||^2&\leq|||\,|A^{t}XB^{1-s}|^r||| \,\,\,|||\,|A^{1-t}XB^{s}|^r|||.
 \end{align*}
  In addition, since the function $G$ is continuous and convex on  $[0,1]\times[0,1]$, it follows that $G$ attains its
maximum at the vertices of the square. Moreover, due to the symmetry there are
two possibilities for the maximum.
 \end{proof}
 Dragomir \cite[p. 316]{drag} proved that
\begin{align}\label{a30}
F\left({a+b\over2}, {c+d\over2}\right)&\leq {1\over2}\left[{1\over b-a}\int_a^bF(x,{c+d\over2})dx+{1\over d-c}\int_c^dF({a+b\over2},y)dy\right]\nonumber\\&\leq{1\over (b-a)(d-c)}
\int_a^b\int_c^d F(x,y)dydx\nonumber\\&\leq{F(a,c)+F(a,d)+F(b,c)+F(b,d)\over4},
 \end{align}
whenever  $F$ is a convex  function   on  $[a,b]\times [c,d]\subseteq \mathbb{R}^2$. Applying inequality \eqref{a30} for the convex function $G(s,t)=|||\,|A^{1-t}XB^{s}|^r|||\,\,\,|||\,|A^{t}XB^{1-s}|^r|||$ on $[0,1]\times[0,1]$ we get the following result.
\begin{corollary}\label{a19}
Suppose that $A, B\in\mathcal{S}_n$, $ X\in \mathcal{M}_n$  and $r\geq0$. Then
\begin{align*}
2|||&\,|A^{1\over2}XB^{1\over2}|^r|||^2\leq {1\over1-2\alpha}\int_\alpha^{1-\alpha}
|||\,\,|A^{s}XB^{1\over2}|^r|||\,\,\,|||\,|A^{1-s}XB^{1\over2}|^r|||\,ds\\&\,\,\,+
{1\over1-2\beta}\int_\beta^{1-\beta} |||\,|A^{1\over2}XB^{1-t}|^r|||\,\,\,|||\,|A^{1\over2}XB^{t}|^r|||\,dt
\\&\leq{2\over(1-2\alpha)(1-2\beta)}\int_\alpha^{1-\alpha}\int_\beta^{1-\beta} |||\,|A^{s}XB^{1-t}|^r|||\,\,\,|||\,|A^{1-s}XB^{t}|^r|||\,dt\,ds\nonumber
\\&\leq|||\,|A^{\alpha}XB^{1-\beta}|^r|||\,\,\,|||\,|A^{1-\alpha}XB^{\beta}|^r|||+
|||\,|A^{1-\alpha}XB^{1-\beta}|^r|||\,\,\,|||\,|A^{\alpha}XB^{\beta}|^r|||
 \end{align*}
 for all $\alpha,\beta\in[0,{1\over2})$ and
\begin{align*}
2|||&\,|A^{1\over2}XB^{1\over2}|^r|||^2\leq {1\over2\alpha-1}\int_{1-\alpha}^\alpha
|||\,|A^{s}XB^{1\over2}|^r||| \,\,\,|||\,|A^{1-s}XB^{1\over2}|^r|||\,ds
\\&\,\,\,\,+{1\over2\beta-1}\int_{1-\beta}^\beta|||\,|A^{1\over2}XB^{1-t}|^r|||\,\,\,|||\,|A^{1\over2}XB^{t}|^r|||\,dt
\\&\leq{2\over(2\alpha-1)(2\beta-1)}\int_{1-\alpha}^{\alpha}\int_{1-\beta}^{\beta} |||\,|A^{s}XB^{1-t}|^r|||\,\,\,|||\,|A^{1-s}XB^{t}|^r|||\,dt\,ds\nonumber
\\&\leq|||\,|A^{\alpha}XB^{1-\beta}|^r|||\,\,\,|||\,|A^{1-\alpha}XB^{\beta}|^r|||
+|||\,|A^{1-\alpha}XB^{1-\beta}|^r|||\,\,\,|||\,|A^{\alpha}XB^{\beta}|^r|||
 \end{align*}
 for all $\alpha,\beta\in({1\over2},1]$.
 \end{corollary}
\begin{proof}
 Let $G(s,t)=|||\,|A^{t}XB^{1-t}|^r||| \,\,\,|||\,|A^{1-t}XB^{t}|^r|||$.
 If we replace $a$ by $\alpha$, $b$ by $1-\alpha$, $c$ by $\beta$ and $d$ by $1-\beta\,\,(\alpha,\beta\in[0,{1\over2}))$ for the convex function $G$ in \eqref{a30} we reach the first inequality and if we replace $a$ by $1-\alpha$, $b$ by $\alpha$,  $c$ by $1-\beta$ and $d$ by  $\beta\,\,(\alpha,\beta\in({1\over2},1])$ in \eqref{a30} we obtain  the second inequality.
\end{proof}
 The spacial case $\alpha=\beta=1$ of Theorem \ref{a19} reads as follows.
\begin{corollary}\label{a21}
Suppose that $A, B\in\mathcal{S}_n$, $ X\in \mathcal{M}_n$   and $r\geq0$. Then
\begin{align*}
2|||\,|A^{1\over2}XB^{1\over2}|^r|||^2&\leq \int_0^1
|||\,|A^{s}XB^{1\over2}|^r||| \,\,\,|||\,|A^{1-s}XB^{1\over2}|^r|||\,ds\\&\,\,\,+\int_0^1
|||\,|A^{1\over2}XB^{1-t}|^r|||\,\,\,|||\,|A^{1\over2}XB^{t}|^r|||\,dt
\\&\leq2\int_0^1\int_0^1|||\,|A^{s}XB^{1-t}|^r|||\,\,\,|||\,|A^{1-s}XB^{t}|^r|||\,dt\,ds
\\&\leq|||\,|AX|^r|||\,\,\,|||\,|XB|^r|||+|||\,|X|^r|||\,\,\,|||\,|AXB|^r|||.
 \end{align*}
\end{corollary}
\section{Further refinements of   the Cauchy-Schwarz inequality }
In this section, we establish some refinements of the Cauchy-Schwarz inequality.
 The following result, derived in the recent papers \cite{krnic2,krnic1}.
 \begin{lemma}\label{krnic23}\cite{krnic2}
 Let $f:[a,b]\rightarrow \mathbb{R}$ be a convex function and $\delta\in[a,b],p\in(0,1)$ be fixed parameters. Then the function $\varphi:[a,b]\rightarrow R$, defined by $$\varphi(t)=(1-p)f(\delta)+pf(t)-f((1-p)\delta+pt)$$ is decreasing on $[a,\delta]$ and is increasing on $[\delta,b]$.
 \end{lemma}
 In the next result, we show a refinement of the right side of inequality \eqref{d41}.

\begin{theorem}\label{a12}
Let  $A, B\in\mathcal{S}_n,$ $ X\in \mathcal{M}_n$, $r\geq0$, $\mu\in[0,1]$, $p\in(0,1)$ and let $|||\,.\,|||$ be any   unitarily invariant norm. Then
\begin{align}\label{lk3x}
|||\,|AX|^r|||\,\,\,|||\,|XB|^r|||&-|||\,|A^{\mu}XB^{1-\mu}|^r|||  \,\,\,|||\,|A^{1-\mu}XB^{\mu}|^r|||\nonumber\\&\geq
{1\over p}\left(f({1-p\over2})-f({1-p\over2}+p\mu)\right)\geq0,
 \end{align}
  where $f(t)=|||\,|A^{t}XB^{1-t}|^r||| \,\,\,|||\,|A^{1-t}XB^{t}|^r|||\,\,(t\in[0,1])$.
 \end{theorem}
 \begin{proof}
Assume that the functions $f(t)=|||\,|A^{t}XB^{1-t}|^r||| \,\,\,|||\,|A^{1-t}XB^{t}|^r|||\,\,(t\in[0,1])$ and  $\varphi(\mu)=(1-p)f\left({1\over2}\right)+pf(\mu)-f\left({1-p\over2}+p\mu\right)\,\,(\mu\in[0,1])$.
Using Lemma \ref{krnic23}, we see that  $\varphi$ is decreasing on $[0,{1\over2}]$ and increasing on $[{1\over2},1]$. Let that $\mu\in[0,{1\over2}]$. Since $\varphi$ is decreasing on $[0,{1\over2}]$, we have $\varphi(0)\geq\varphi(\mu)$, that is,
\begin{align*}
pf(0)-f\left({1-p\over2}\right)\geq pf(\mu)-f\left({1-p\over2}+p\mu\right),
 \end{align*}
 whence \begin{align}\label{mn65}
f(0)-f(\mu)\geq {1\over p}\left[ f\left({1-p\over2}\right)-f\left({1-p\over2}+p\mu\right)\right],
 \end{align}
 which yields desired inequality. Note, the right hand side of \eqref{mn65} is decreasing and ${1-p\over2}+p\mu\geq {1-p\over2}$. Now let $\mu\in[{1\over2},1]$. So $0\leq1-\mu\leq{1\over2}$. By the symmetry property of \eqref{mn65}  with respect to $\mu$ , if we replace $\mu$ by $1-\mu$, then
\begin{align*}
f(0)-f(1-\mu)\geq {1\over p}\left[ f\left({1-p\over2}\right)-f\left({1-p\over2}-p\mu\right)\right],
 \end{align*}
 which is reduce to \eqref{lk3x} since $f(1-\mu)=f(\mu),\,\,(\mu\in[0,1])$.
 \end{proof}
 By the same strategy as in the proof of Theorem \ref{a12}, we get a refinement of the left side  inequality \eqref{d41}.
 \begin{theorem}\label{a12}
Let  $A, B\in\mathcal{S}_n,$ $ X\in \mathcal{M}_n$, $r\geq0$, $\mu\in[0,1]$, $p\in(0,1)$ and let $|||\,.\,|||$ be any   unitarily invariant norm. Then
\begin{align*}
|||\,|A^{\mu}XB^{1-\mu}|^r|||  \,\,\,&|||\,|A^{1-\mu}XB^{\mu}|^r|||-|||\,|A^{1\over2}XB^{1\over2}|^r|||^2\\&\geq
{1\over p}\left(f({1-p\over2}+p\mu)-|||\,|A^{1\over2}XB^{1\over2}|^r|||^2\right)\geq0,
 \end{align*}
  where $f(t)=|||\,|A^{t}XB^{1-t}|^r||| \,\,\,|||\,|A^{1-t}XB^{t}|^r|||\,\,(t\in[0,1])$.
 \end{theorem}
\section{Some inequalities  involving numerical radius}
 In this section we show  inequalities  involving Heinz type numerical radius.
A continuous real valued function $f$ defined on an interval
$(a,b)$ with $a\geq0$ is called Kwong function if the matrix
\begin{align*}
\left({f(a_i)+f(a_j)\over a_i+a_j}\right)_{i,j=1}^n
\end{align*}
 is positive semidefinite   for any distinct real numbers $a_1,\cdots, a_n$ in $(a,b)$.

\begin{lemma}\cite[Corollary 4]{Okubo}\label{Okubo1}
Let  $A=[a_{ij}]\in\mathcal{M}_n$ be positive semidefinite. Then
\begin{align*}
\|S_A\|_\omega=\max_i a_{ii}.
 \end{align*}
\end{lemma}
\begin{lemma}\cite[Theorem 3.4]{Zhang1}\label{shour}
(Spectral Decomposition) Let $A\in\mathcal{M}_n$ with eigenvalues
$\lambda_1, \lambda_2, \cdots, \lambda_n$. Then $A$ is normal if and
only if there exists a unitary matrix $U$ such that
\begin{align*}
U^*AU={\rm diag}(\lambda_1, \lambda_2, \cdots, \lambda_n).
 \end{align*}
 In particular, $A$ is positive definite if and only if the $\lambda_j\,\,(1\leq j \leq n)$ are positive.
\end{lemma}

\begin{theorem}\label{a18}
Suppose that  $A\in\mathcal{P}_n$, $X\in\mathcal{M}_n$, $\alpha\in[0,1]$ and ${f\over g}$ be a Kwong function such that ${f(t) g(t)}\leq t\,\,(t\geq0)$. Then
\begin{align*}
\omega(f(A)Xg(A)+g(A)Xf(A))\leq \omega(AX+XA)
 \end{align*}
\end{theorem}\begin{proof}
 Applying Lemma \ref{shour}, we can assume that $A={\rm diag}(a_1, a_2, \cdots, a_n)$ is diagonalize,  where $a_j\,\,(j=1,2, \cdots, n)$ are positive numbers. Let $Z=[z_{ij}]\in\mathcal{M}_n$ with the entries $z_{ij}={f(a_i)g(a_i)+f(a_j)g(a_j)\over a_i+a_j}\,\,(1\leq i,j\leq n)$. Since ${f\over g}$ is a Kwong function,
 \begin{align*}
Z=S\left({f(a_i)g^{-1}(a_i)+f(a_j)g^{-1}(a_j)\over a_i+a_j}\right)_{i,j=1}^nS
 \end{align*}
 is positive semidefinite where $S={\rm diag}\left(g(a_1),\cdots,g(a_n)\right)$.
 It  follows from  Lemma \ref{Okubo1} that
 \begin{align*}
\|S_Z\|_\omega=\max_i z_{ii}={f(a_i)g(a_i)\over a_i} \leq1,
 \end{align*}
 or equivalently, ${\omega(Z\circ X)\over \omega(X)}\leq 1\,\,(0\neq X \in\mathcal{M}_n)$. Let   $E=[{1\over a_i+a_j}]$ and $ D=[f(a_i)g(a_i)+f(a_j)g(a_j)]\in \mathcal{M}_n$. Hence
  \begin{align*}
\omega(D\circ E\circ X)=\omega(Z\circ X)\leq \omega(X)\qquad(X \in\mathcal{M}_n).
 \end{align*}
  Let  the matrix $C$ be the entrywise inverse of $E$, i.e., $C\circ E=J$. Thus $\omega(D\circ X)\leq \omega(C\circ X)\,\,(X \in\mathcal{M}_n)$.  Hence
\begin{align*}
\omega(f(A)Xg(A)+g(A)Xf(A))\leq \omega(AX+XA).
 \end{align*}
 \end{proof}
 Using $f(t)=t^\alpha$  and $g(t)=t^{1-\alpha}$ in Theorem \ref{a18}
 we get  the following Heinz type inequality in the following result.
\begin{corollary}
Suppose that $A\in\mathcal{P}_n$, $X\in\mathcal{M}_n$ and $\alpha\in[0,1]$. Then
\begin{align*}
\omega(A^{\alpha}XA^{1-\alpha}+A^{1-\alpha}XA^{\alpha})\leq \omega(AX+XA)
 \end{align*}
\end{corollary}
Kwong \cite{kwong} showed that the set Kwong functions on $(0,\infty)$ includes
all non-negative operator monotone functions $f$ on $(0,\infty)$.
\begin{example}
The function $f(t)=\log(t+1)$  is operator monotone on the interval $(0,\infty)$ \cite{abc}. If $g(t)={t\over f(t)}$, then,   by Theorem \ref{a18},  for every unitarily invariant norm $|||\,.\,|||$, $A\in\mathcal{P}_n$ and $X\in\mathcal{M}_n$ we have
\begin{align*}
\omega\left(\log(A+1)XA\log(A+1)^{-1}+A\log (A+1)^{-1}X\log(A+1)\right)\leq \omega(AX+XA).
 \end{align*}
\end{example}


\end{document}